\documentclass{amsart}
\usepackage[T1]{fontenc}
\usepackage[latin1]{inputenc}
\usepackage[english]{babel}
\usepackage{mathptmx}
\usepackage{graphicx,subfigure}
\usepackage{amsmath}
\usepackage{amssymb}
\usepackage{amsthm}
 \usepackage{enumerate}
\usepackage{theoremref}
 \usepackage{amssymb,amsfonts}
 \theoremstyle{plain}
 \newtheorem{theo}{Theorem}[section]
 \newtheorem{lem}{Lemma}[section]
 \newtheorem{cor}{Corollary}[section]
\newtheorem{defi}{Definition}[section]

\newtheorem{ex}{\normalfont\it Example}[section]
\newfont{\bbold}{bbold10 scaled\magstep1}
\newcommand{\Id}{1\kern-0.25em{\rm l}}
\newcommand{\D}{\displaystyle}

\newcommand{\real}{\mathbb{R}}
\newcommand{\mlt}{\circ}
\newcommand{\mltl}{\ast}

\providecommand{\E}{\operatorname{E}}
\providecommand{\N}{\mathbb{N}}

\providecommand{\C}{\mathbb{C}}
\newcommand{\frh}{\mathfrak{h}}
\newcommand{\frr}{\mathfrak{r}}
\newcommand{\frd}{\mathfrak{d}}
\newcommand{\frg}{\mathfrak{g}}
\newcommand{\Tau}{\mathcal{T}}
\newcommand{\arsinh}{\operatorname{arsinh}}
\newcommand{\bO}{\ensuremath{\mathcal{O}}}
\newcommand{\Ih}{I^h}
\newcommand{\I}[1]{I_{(#1)}}
\newcommand{\g}{g_{1}}
\newcommand{\f}{g_{0}}
\usepackage{pstricks,pst-node,pst-tree}
\newcommand{\dtree}[1]{\pstree{#1}}
\psset{levelsep=-12pt,nodesep=0pt,treesep=8pt,radius=2pt,tnsep=1pt}
\newcommand{\tn}{\TC*}
\newcommand{\tnr}[1]{\TC*~[tnpos=r]{\ensuremath{\scriptstyle #1}}}
\newcommand{\tnl}[1]{\TC*~[tnpos=l]{\ensuremath{\scriptstyle #1}}}
\newcommand{\node}[1]{\bullet_{#1}}
\usepackage{scrtime}
\newlength{\figwidth}
\newcommand{\Farbbild}[1]{#1}
\title{Stochastic B--series analysis of iterated Taylor methods}
\author{Kristian Debrabant}
\address[Kristian Debrabant]{Technische Universit\"{a}t Darmstadt, Fachbereich Mathematik, Dolivo\-stra{\ss}e 15,
D-64293 Darmstadt, Germany}
\email{debrabant@mathematik.tu-darmstadt.de}
\author{Anne Kv{\ae}rn{\o}}
\address[Anne Kv{\ae}rn{\o}]{
Department of Mathematical Sciences, Norwegian University of Science and Technology, N-7491 Trondheim, Norway}
\email{anne.kvarno@math.ntnu.no}

\begin{document}
\maketitle
\begin{abstract}
For stochastic implicit Taylor methods that use an iterative scheme to compute their numerical solution, stochastic B--series and corresponding growth functions are constructed. From these, convergence results based on the order of the underlying Taylor method, the choice of the iteration method, the predictor and the number of iterations, for It\^{o} and Stratonovich SDEs, and for weak as well as strong convergence are derived. As special case, also the application of Taylor methods to ODEs is considered. The theory is supported by numerical experiments.
\end{abstract}
\psset{levelsep=-12pt,nodesep=0pt,treesep=8pt,radius=2pt,tnsep=1pt}
\renewcommand{\tn}{\TC*}
\renewcommand{\tnr}[1]{\TC*~[tnpos=r]{\ensuremath{\scriptstyle #1}}}
\renewcommand{\tnl}[1]{\TC*~[tnpos=l]{\ensuremath{\scriptstyle #1}}}
\setlength{\figwidth}{0.485 \textwidth}
\section{Introduction}
Besides stochastic Runge--Kutta methods, one important class of schemes to approximate the solution of stochastic differential equations (SDEs) are stochastic Taylor methods. As in the deterministic setting \cite{barrio05pot}, they are especially suitable for problems with not too high dimension, and here especially in the case of strong approximation, because weak approximation of low-dimensional problems can often be done more efficiently by numerically solving the corresponding deterministic PDE problem obtained by applying the Feynman-Kac formula.
For solving stiff SDEs, implicit methods have to be considered, as is illustrated in the following two examples.
\begin{ex}[see \cite{higham00msa}] Consider the linear It\^{o}-SDE
\begin{equation}\label{eq:gbm}
dX(t)=\mu X(t)~dt+\sigma X(t)~dW(t),\quad X(0)=x_0,
\end{equation}
with $\mu,\sigma\in\C$. We assume that the exact solution is mean-square stable, i.\,e.\
\[\lim_{t\to\infty}\E(|X(t)|^2)=0,\]
which is the case if and only if $2\Re\mu+|\sigma|^2<0.$ To achieve that also the numerical
approximation $Y_{n}$ obtained with the (explicit) Euler-Maruyama scheme with step size $h$ is mean-square stable, i.\,e.\
$\lim_{n\to\infty}\E(|Y_n|^2)=0$, we have to restrict the step size according to $h<h_0:=-\frac{2\Re\mu+|\sigma|^2}{|\mu|^2}$, whereas
for $h>h_0$ the numerical approximations explode, $\lim_{n\to\infty}\E(|Y_n|^2)=\infty$. In contrast to this, the semi-implicit Euler
scheme is mean-square stable without any step size restriction. For a numerical affirmation, see Figure \ref{fig:implmotiv1}.
\end{ex}
\begin{figure}[tbp]
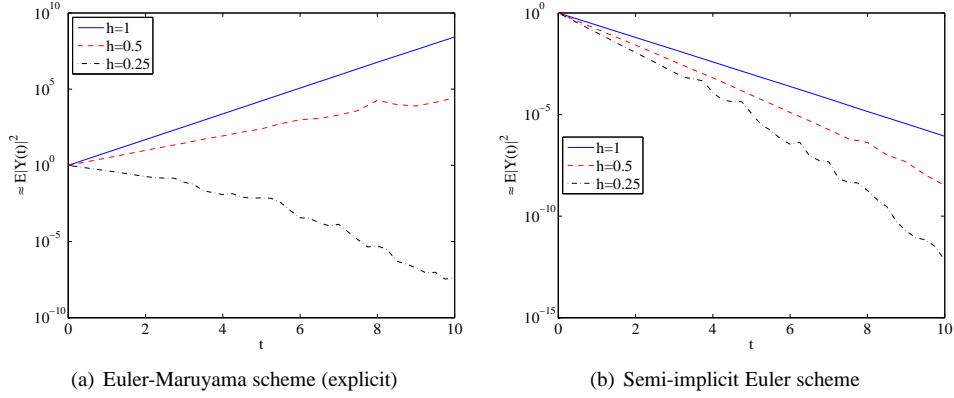
 \subfigure[Euler-Maruyama scheme (explicit)]{\includegraphics[width=\figwidth]{\Farbbild{ExplEuler}}} \hspace*{\fill}
\subfigure[Semi-implicit Euler scheme]{\includegraphics[width=\figwidth]{\Farbbild{ImplEuler}}}
 \caption{\label{fig:implmotiv1}Approximation results for the linear test equation \eqref{eq:gbm} with $\mu =-3$, $\sigma =\sqrt3$, and $x_0=1$ by
Euler-Maruyama (explicit) and semi-implicit Euler scheme. Here, $\E(|Y(t)|^2)$ is approximated as mean over $10^6$ simulations. The
explicit scheme is only stable for appropriate step sizes, the semi-implicit scheme is stable for all step sizes.}
 \end{figure}
\begin{ex}
Consider the following stochastic Van der Pol equation,
\begin{align*}
dX_1(t)&=X_2(t)~dt,\\
dX_2(t)&=\left(\mu(1-X_1(t)^2)X_2(t)-X_1(t)\right)~dt+\theta(1-X_1(t)^2)X_2(t)~dW(t),\\
X_1(0)&=x_{0,1}, \quad X_2(0)=x_{0,2}.
\end{align*}
Application of the explicit Milstein scheme (see Example \ref{ex:SIMilstein} with $\alpha=0$ and $\beta=0$) with step-size $h=0.05$ to approximate a solution path leads to an explosion of the approximation, see Figure \ref{fig:implmotiv3}, whereas application of the semi-implicit Milstein scheme, given by substituting $g_0(Y_{n})$ by $g_0(Y_{n+1})$ ($\alpha=1$ and $\beta=0$ in Example \ref{ex:SIMilstein}), yields (for the same Brownian path) the result of Figure \ref{fig:implmotiv2}.
\end{ex}
\begin{figure}[tbp]
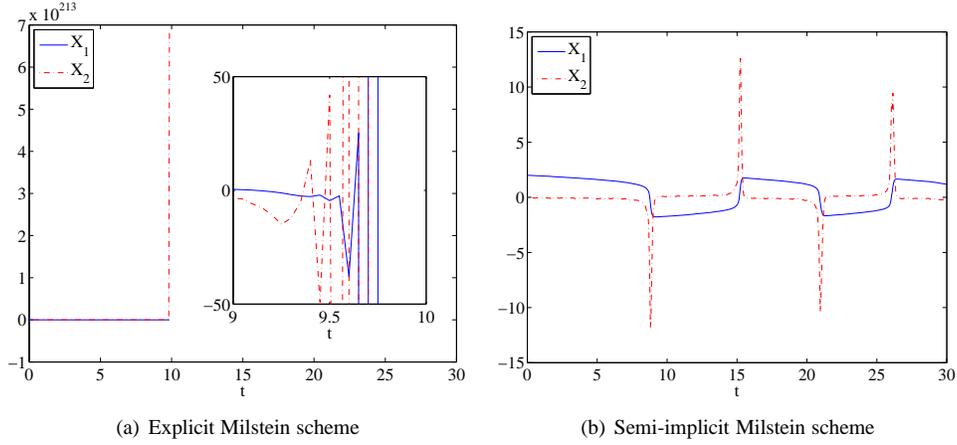

\subfigure[\label{fig:implmotiv3}Explicit Milstein scheme]{\includegraphics[width=\figwidth]{\Farbbild{ExplMilstein}}}
\hspace*{\fill}
\subfigure[\label{fig:implmotiv2}Semi-implicit Milstein scheme]{\includegraphics[width=\figwidth]{\Farbbild{ImplMilstein}}}
 \caption{Approximation of Van der Pol equation with $\mu = 10$, $\theta = 1$, $x_{0,1}=2$, and $x_{0,2}=0$ by the explicit and semi-implicit Milstein scheme (using the same Brownian path) with step-size $h=0.05$. The explicit scheme suffers from heavy stability problems and aborts.}
 \end{figure}
Implicit stochastic Taylor methods have been considered both for strong \cite{kloeden99nso,tian01itm} and weak \cite{kloeden99nso} approximation. For these methods, the approximation values are only given implicitly. However, in practice these implicit equations are solved by iterative schemes like simple iteration or Newton iteration. The ``exact numerical'' solution can be written in terms of B--series \cite{debrabantXXcos}. As we will prove in this paper, so can the iterated solution. Moreover, for each iteration scheme in question, we will define a growth function. Briefly explained, when the exact numerical and the $k$ times iterated solutions are both written in terms of B--series, then all terms of these series for which the growth function has a value not greater than $k$ coincide.
Thus the growth functions give a quite exact description of the development of the iterations. B--series and corresponding growth functions for iterated solutions have been derived for Runge--Kutta methods applied to deterministic ordinary differential equations \cite{jackson96aao}, differential algebraic equations \cite{jackson94tuo}, and SDEs \cite{debrabant08bao}. Somewhat surprisingly, the growth functions are exactly the same in all these cases, and, as we will show in this paper, this also holds for implicit Taylor methods.

The outline of the paper is as follows: First, we will give the SDE to be solved and the iterated Taylor methods used for its approximation. In Section \ref{bsa2:sec:prelim} stochastic B--series are introduced and some useful preliminary results are presented. The main results of the paper can be found in Section \ref{bsa2:sec:iter}, where the B--series of the iterated solutions are developed and the before mentioned growth functions derived.
In Section \ref{bsa2:sec:generalconvergenceresults}, these findings are interpreted in terms of the order of the overall scheme, giving concrete results on the order of the considered methods depending on the kind and number of iterations, both for SDEs and ODEs. Contrary to the results obtained for Runge--Kutta methods \cite{debrabant08bao}, the order of the iteration error is shown to be independent on whether It\^{o} or  Stratonovich SDEs, weak or strong convergence is considered.
Finally, in Section \ref{bsa2:sec:numerics} we present several numerical examples to support our theoretical findings.

Let $(\Omega,\mathcal{A},\mathcal{P})$ be a probability space. We
denote by $(X(t))_{t \in I}$ the stochastic process which is the
solution of a $d$-dimensional SDE defined by
\begin{equation}\label{bsa2:eq:SDEdiff}
dX(t)=g_0(X(t))dt+\sum_{l=1}^{m}g_{l}(X(t))\star dW_l(t),\quad X(t_0)=x_{0},
\end{equation}
with an $m$-dimensional Wiener process $(W(t))_{t \geq 0}$ and
$I=[t_0,T]$. As usual, \eqref{bsa2:eq:SDEdiff} is construed as abbreviation of
\begin{equation}\label{bsa2:SDE}
X(t)=x_{0}+\int_{t_{0}}^{t}g_0(X(s))ds+\sum_{l=1}^{m}\int_{t_{0}}^{t}g_{l}(X(s))\star
dW_l(s).
\end{equation}
The integral w.\,r.\,t.\ the Wiener process has to be
interpreted e.\,g.\ as It\^{o} integral with $\star dW_l(s)=dW_l(s)$ or
as Stratonovich integral with $\star dW_l(s)=\circ dW_l(s)$. We assume that the Bo\-rel-mea\-surable coefficients $g_l :
\mathbb{R}^d \rightarrow \mathbb{R}^{d}$ are sufficiently differentiable and chosen such that SDE \eqref{bsa2:SDE} has a unique solution.

To simplify the presentation, we define $W_0(s)=s$, so that
\eqref{bsa2:SDE} can be written as
\begin{equation}\label{bsa2:SDE_new}
X(t)=x_{0}+\sum_{l=0}^{m}\int_{t_{0}}^{t}g_{l}(X(s))\star
dW_l(s).
\end{equation}
Let a discretization ${\Ih} = \{t_0, t_1, \ldots, t_N\}$ with
$t_0 < t_1 < \ldots < t_N =T$ of the time interval $I$ with step
sizes $h_n = t_{n+1}-t_n$ for $n=0,1, \ldots, N-1$ be given.
Now, we consider the general class of stochastic Taylor methods given by
$Y_0=x_0$ and
\begin{eqnarray}\label{bsa2:eq:STM}
Y_{n+1}=B(\Phi_{ex},Y_n;h_n)+B(\Phi_{im},Y_{n+1};h_n)
\end{eqnarray}
for $n=0,1,\dots,N-1$ with $Y_n=Y(t_n)$, $t_n\in {\Ih}$,
$\Phi_{ex}(\emptyset)\equiv1$, $\Phi_{im}(\emptyset)\equiv0$.
\begin{ex}\label{ex:SIMilstein}
Consider the family of Milstein schemes applied to an It\^{o} SDE with one-dimensional noise,
\begin{align}
\label{bsa2:eq:sim}
Y_{n+1}&=Y_n+h_n\big((1-\alpha)g_0(Y_n)+\alpha g_0(Y_{n+1})\big) \\
       & +I_{(1),h_n}\big((1-\beta)g_1(Y_n)+\beta g_1(Y_{n+1})\big)
       +(I_{(1,1),h_n}-\beta I_{(1),h_n}^2)[g_1'g_1](Y_n). \nonumber
\end{align}
Here,
\begin{align*}
I_{(1),h_n}&=W(t_{n+1})-W(t_{n})=\Delta W_n,\\
I_{(1,1),h_n}&=\int_{t_n}^{t_{n+1}}W(s) dW(s)=\frac{1}{2}(\Delta W_n^2-h_n),
\end{align*}
and the parameters $\alpha, \beta \in [0,1]$ indicate the degree of implicitness. When $\alpha=\beta=0$ we have the explicit Milstein scheme, with $\alpha\not=0$, $\beta=0$ a semi-implicit scheme. In all cases, the method
\eqref{bsa2:eq:sim} can be written in the form \eqref{bsa2:eq:STM} with
\begin{align*}
   B(\phi_{ex},Y_n;h_n) &= x_0 + h_n(1-\alpha)g_0(Y_n)
   + I_{(1),h_n}(1-\beta)g_1(Y_n)
   \\&+\big(I_{(1,1),h_n}-\beta I_{(1),h_n}^2\big)\,(g_1'g_1)(Y_n),
   \\
   B(\phi_{im},Y_{n+1};h_n) &= h_n \alpha g_0(Y_{n+1})
   + I_{(1),h_n}\beta g_1(Y_{n+1}).
\end{align*}
The terms $\Phi_{ex}$ and $\Phi_{im}$ refer to the time- and method-dependent part of each term: In this case
\begin{align*}
\Phi_{ex}(\bullet_0) &= h_n(1-\alpha), & \Phi_{im}(\bullet_0)&=h_n \alpha,\\
\Phi_{ex}(\bullet_1) &= I_{(1),h_n}(1-\beta), &\Phi_{im}(\bullet_1) &= I_{(1),h_n}\beta,\\
\Phi_{ex}([\bullet_1]_1]) &= I_{(1,1),h_n}-\beta I_{(1),h_n}^2.
\end{align*}
The notation will be explained in detail in Section \ref{bsa2:sec:prelim}.
\end{ex}

What the general method \eqref{bsa2:eq:STM} concerns, application of an iterative Newton-type method yields
\begin{eqnarray}\label{bsa2:eq:SRKiter}
Y_{n+1,k+1}=B(\Phi_{ex},Y_n;h_n)+B(\Phi_{im},Y_{n+1,k};h_n)+J_k(Y_{n+1,k+1}-Y_{n+1,k})
\end{eqnarray}
with some approximation $J_k$ to the Jacobian of
$B(\Phi_{im},Y_{n+1,k};h_n)$ and a predictor $Y_{n+1,0}$. In the following we assume that \eqref{bsa2:eq:SRKiter} can be solved uniquely at least for sufficiently small $h_n$. To simplify the presentation, we assume further that all step sizes are constant, $h_n=h$.

For the approximation $J_k$ there exist several common choices.
 If we choose $J_k$ to be the exact Jacobian $\partial_2B(\Phi_{im},Y_{n+1,k};h)$, then we obtain the classical Newton iteration method for solving \eqref{bsa2:eq:STM},  with quadratic convergence.  It will be denoted in the following as full Newton iteration. If we choose instead $J_k = \partial_2B(\Phi_{im},Y_{n};h)$, then we obtain the so called modified Newton iteration method, which is only linearly convergent. Here, $J_k$ is independent of the iteration number $k$, thus its computation is much cheaper and simpler than in the full Newton iteration case. The third and simplest possibility is to choose $J_k$ equal to zero. In this case we don't even have to solve a linear system for $Y_{n+1,k+1}$. This iteration method is called simple iteration method or predictor corrector method. Its disadvantage is that for stiff systems it requires very small step sizes to converge.

For most stiff problems and problems with additive noise, the use of semi-implicit methods suffices. As will be demonstrated in Section \ref{bsa2:sec:generalconvergenceresults} these have the advantage that less iterations are required to obtain the correct order of the underlying
method.

\section{Stochastic B--series}\label{bsa2:sec:prelim}
B--series, symbolized by $B(\phi,x_0;h)$, for SDEs were first constructed by Burrage and Burrage \cite{burrage96hso,burrage00oco} to study strong convergence in the Stratonovich case. In the following years, this approach has been further developed by several authors to study weak and strong convergence, for the It\^{o} and the Stratonovich case, see e.\,g.\ \cite{debrabant08bao} for an overview. A uniform theory for the construction of stochastic B--series has been presented in \cite{debrabant08bao}, in \cite{debrabantXXcos} this approach has been used to construct order conditions for implicit Taylor methods. Following the exposition of these two papers, we define in this section stochastic B--series and present some preliminary results that will be used later.

\subsection{Some useful definitions and preliminary results}
\begin{defi}[Trees] \thlabel{bsa2:def:trees}
    The set of $m+1$-colored, rooted trees \[T=\{\emptyset\}\cup T_0 \cup T_1 \cup
    \dots \cup T_m\] is recursively defined as follows:
    \begin{description}
    \item[a)] The graph $\bullet_l=[\emptyset]_l$ with only one vertex
        of color $l$ belongs to $T_l$.
    \end{description}
    Let $\tau=[\tau_1,\tau_2,\dots,\tau_{\kappa}]_l$ be the tree
    formed by joining the subtrees
    $\tau_1,\tau_2,\dots,\tau_{\kappa}$ each by a single branch to a
    common root of color $l$.
    \begin{description}
    \item[b)] If $\tau_1,\tau_2,\dots,\tau_{\kappa} \in T$ then
        $\tau=[\tau_1,\tau_2,\dots,\tau_{\kappa}]_l \in T_l$.
    \end{description}
\end{defi}
Thus, $T_l$ is the set of trees with an $l$-colored root, and $T$ is
the union of these sets.

\begin{defi}[Elementary differentials] \thlabel{bsa2:def:eldiff}
    For a tree $\tau \in T$ the elementary differential is
    a mapping $F(\tau):\real^d \rightarrow \real^d$ defined
    recursively by
    \begin{description}\setlength{\itemsep}{4mm plus2mm minus2mm}
    \item[a)] $F(\emptyset)(x_0)=x_0$,
    \item[b)] $F(\bullet_l)(x_0)=g_l(x_0)$,
    \item[c)] If $\tau=[\tau_1,\tau_2,\dots,\tau_{\kappa}]_l \in T_l$
        then
        \[
        F(\tau)(x_0)=g_l^{(\kappa)}(x_0)
        \big(F(\tau_1)(x_0),F(\tau_2)(x_0),\dots,F(\tau_{\kappa})(x_0)\big). \]
    \end{description}
\end{defi}

With these definitions in place, we can define the stochastic B--series:
\begin{defi}[B--series] \thlabel{bsa2:def:bseries}
 A (stochastic) B--series is
    a formal series of the form
    \[ B(\phi,x_0; h) = \sum_{\tau \in T}
    \alpha(\tau)\cdot\phi(\tau)(h)\cdot F(\tau)(x_0), \] where
    $\phi:T \rightarrow \Xi:=\big\{\{\varphi(h)\}_{h\geq0}:\;
    \varphi(h):~\Omega\to\real\text{ is Borel-measurable }\forall h\geq0
    \big\}$ assigns to each tree a random variable,  and
    $\alpha: T\rightarrow \mathbb{Q}$ is given by
    \begin{align*}
        \alpha(\emptyset)&=1,&\alpha(\bullet_l)&=1,
        &\alpha(\tau=[\tau_1,\dots,\tau_{\kappa}]_l)&=\frac{1}{r_1!r_2!\cdots
            r_{q}! } \prod_{j=1}^{\kappa} \alpha(\tau_j),
    \end{align*}
    where $r_1,r_2,\dots,r_{q}$ count equal trees among
    $\tau_1,\tau_2,\dots,\tau_{\kappa}$.
\end{defi}

Note that $\alpha(\tau)$ is the inverse of the order of the automorphism group of $\tau$.

To simplify the presentation, in the following we assume that all elementary differentials exist and all considered B--series converge. Otherwise, one has to consider truncated B--series and discuss the remainder term.

The next lemma is proved in \cite{debrabant08bao}.
\begin{lem} \thlabel{bsa2:lem:f_y} If $Y(h)=B(\phi, x_0; h)$ with $\phi(\emptyset) \equiv 1$ is some
    B--series and $f\in C^{\infty}(\real^d,\real^{\hat{d}})$, then
    $f(Y(h))$ can be written as a formal series of the form
    \begin{equation}
        \label{bsa2:eq:f}
        f(Y(h))=\sum_{u\in U_f} \beta(u)\cdot \psi_\phi(u)(h)\cdot G(u)(x_0),
    \end{equation}
    where
    \begin{description}\setlength{\itemsep}{4mm plus2mm minus2mm}
    \item[a)] $U_f$ is a set of trees derived from $T$ as follows: $[\emptyset]_f \in U_f$, and if\\
        $\tau_1,\tau_2,\dots,\tau_{\kappa}\in T\setminus\{\emptyset\}$ then
        $[\tau_1,\tau_2,\dots,\tau_{\kappa}]_f\in U_f$,
    \item[b)] $G([\emptyset]_f)(x_0)=f(x_0)$ and \\
        $\mbox{}\;\, G(u=[\tau_1,\dots,\tau_{\kappa}]_f)(x_0) =
        f^{(\kappa)}(x_0)
        \big(F(\tau_1)(x_0),\dots,F(\tau_{\kappa})(x_0)\big)$,
    \item[c)] $\beta([\emptyset]_f)=1$ and $\D
        \beta(u=[\tau_1,\dots,\tau_{\kappa}]_f)
        =\frac{1}{r_1!r_2!\cdots r_{q}!}\prod_{j=1}^{\kappa}
        \alpha(\tau_{{j}})$, \\where $r_1,r_2,\dots,r_{q}$ count
        equal trees among $\tau_1,\tau_2,\dots,\tau_{\kappa}$,
    \item[d)] $\psi_\phi([\emptyset]_f)\equiv1$ and
        $\psi_\phi(u=[\tau_1,\dots,\tau_{\kappa}]_f)(h) =
        \prod_{j=1}^{\kappa} \phi(\tau_j)(h)$.
    \end{description}
\end{lem}
For notational convenience, in the following the $h$-dependency of the weight functions and the $x_0$-dependency of the elementary differentials will be suppressed whenever this is unambiguous, so $\Phi(\tau)(h)$ will be written as $\Phi(\tau)$ and $F(\tau)(x_0)$ as $F(\tau)$.

The next step is to present the composition rule for B--series. In the deterministic  case, this is e.\,g.\ given by \cite{hairer02gni}, using ordered trees. The same rule applies for multicolored trees, as in the stochastic case. But it is also possible to present the result without relying on ordered trees, as done in \cite{debrabantXXcos}, and this is the approach that will be used in the following.

Consider triples $(\tau,\vartheta,\omega)$ consisting of some $\tau\in T$,  a subtree $\vartheta$ sharing the root with $\tau$, and a remainder multiset $\omega$ of trees left over when $\vartheta$ is removed from $\tau$.  We also include the empty tree as a possible subtree, in which case the triple becomes $(\tau,\emptyset,\tau)$.
\begin{ex} \thlabel{bsa2:ex:triples}
Two examples of such triples are
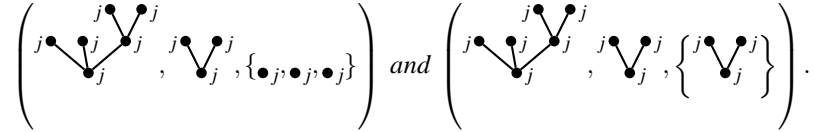
\[
\left(\dtree{\tnr{j}}{\tnl{j} \tnr{j} \dtree{\tnr{j}}{\tnl{j} \tnr{j}}},
       \dtree{\tnr{j}}{\tnl{j} \tnr{j}},
       \{\tnr{j},\tnr{j},\tnr{j}\}
\right)
\text{ and }
\left(\dtree{\tnr{j}}{\tnl{j} \tnr{j} \dtree{\tnr{j}}{\tnl{j} \tnr{j}}},
       \dtree{\tnr{j}}{\tnl{j} \tnr{j}},
       \left\{\dtree{\tnr{j}}{\tnl{j} \tnr{j}}\right\}
\right).
\]
\end{ex}
So, for the same $\tau$ and $\vartheta$ there might be different $\omega$'s.

We next define
$ST(\tau)$ as the set of all possible subtrees of $\tau$ together with the corresponding  remainder multiset $\omega$, that is, for each $\tau\in T\backslash \emptyset$ we have
\begin{align*}
ST(\node{l}) &= \{ (\emptyset,\node{l}), (\node{l},\emptyset) \}, \\
ST(\tau =[\tau_1,\dotsc,\tau_{\kappa}]_l) &=
\big\{(\vartheta,\omega)\;:\; \vartheta=[\vartheta_1,\dotsc,\vartheta_{\kappa}]_l, \quad
   \omega=\{\omega_1,\dotsc,\omega_{\kappa}\}, \\ &\qquad
   (\vartheta_i,\omega_i) \in ST(\tau_i), \quad i=1,\dots,\kappa \big\} \cup
   (\emptyset,\tau).
\end{align*}
We also have to take care of possible equal terms in the formula presented below. This is done as follows:
For a given triple $(\tau,\vartheta,\omega)$ write first $\vartheta=[\vartheta_1,\dotsc,\vartheta_{\kappa_{\vartheta}}]_l = [\bar{\vartheta}_1^{r_1},\dotsc,\bar{\vartheta}_{q}^{r_q}]_l$, where the latter only expresses that $\vartheta$ is composed by $q$ different nonempty trees, each appearing $r_i$ times, hence $\sum_{i=1}^qr_i=\kappa_{\vartheta}$.
Let $\tau=[\tau_1,\dotsc,\tau_{\kappa}]_l$. For $i=1,\dots,q$, each $\bar{\vartheta}_i$ is a subtree of some of the $\tau_j$'s, with corresponding remainder multisets $\omega_j$. Assume that there are exactly $p_i$ different such triples $(\bar{\tau}_{ik},\bar{\vartheta}_i,\bar{\omega}_{ik})$ each appearing exactly $r_{ik}$ times so that $\sum_{k=1}^{p_i}r_{ik}=r_i$.
Finally, let $\bar\delta_k\in\omega$ be the distinct trees with multiplicity $s_k$, $k=1,\dots,p_0$, of the remainder multiset which are directly connected to the root of $\tau$.
Then, $\tau$ can be written as
\begin{equation}
\label{bsa2:eq:alltrees}
\tau=[\bar\delta_1^{s_1},\dotsc, \bar\delta_{p_0}^{s_{p_0}},
    \bar{\tau}_{11}^{r_{11}},\dotsc,\bar{\tau}_{1p_1}^{r_{1p_1}} ,\dotsc,
    \bar{\tau}_{q1}^{r_{q1}},\dotsc,\bar{\tau}_{qp_q}^{r_{qp_q}}]_l=
    [ \bar{\tau}_1^{R_1},\dotsc,\bar{\tau}_{Q}^{R_Q}]_l,
\end{equation}
where the rightmost expression above indicates that $\tau$ is composed by $Q$ different trees each appearing $R_i$ times.

With these definitions, we can state the following theorem, proved in \cite{debrabantXXcos}:
\begin{theo}[Composition of B--series] \thlabel{bsa2:thm:comp}
Let \mbox{$\phi_x,\phi_y\;:\;T \rightarrow \Xi$} and $\phi_x(\emptyset)\equiv 1$. Then the B--series $B(\phi_x,x_0;h)$ inserted into $B(\phi_y,\cdot;h)$ is again a B--series,
\[ B(\phi_y,B(\phi_x,x_0;h);h)=B(\phi_x\mlt \phi_y,x_0;h), \]
where
\[ \left(\phi_x\mlt \phi_y \right)(\tau) = \sum_{(\vartheta,\omega)\in ST(\tau)} \gamma(\tau,\vartheta,\omega)\left( \phi_y(\vartheta)\prod_{\delta \in \omega} \phi_x(\delta)\right)(\tau),
\]
$\gamma(\emptyset,\emptyset,\emptyset)=1$, and
\[
\gamma(\tau,\vartheta,\omega)=\frac{R_1!\dotsm R_Q!}{s_1!\dotsm s_{p_0}! r_{11}!\dotsm r_{qp_q}!}
\prod_{i=1}^q\prod_{k=1}^{p_i} \gamma(\bar\tau_{ik},\bar\vartheta_i,\bar\omega_{ik})^{r_{ik}}
\]
for $\tau$ given by \eqref{bsa2:eq:alltrees}.
\end{theo}
The combinatorial term $\gamma$ gives the number of equal terms that will appear if the composition rule using ordered trees is preferred.

In general, the composition law is not linear, neither is it associative.
It is, however, linear in its second operand. Further,
if both $\phi_x(\emptyset)=\phi_y(\emptyset)\equiv1$, then the composition law
  can be turned into a group operation (Butcher group, see
  \cite{butcher72aat,hairer02gni,hairer74otb} for the deterministic case): The
  inverse element $\phi^{-1}(\tau)$ can be recursively computed by
\begin{equation} \label{bsa2:eq:id_element}
  (\phi\mlt\phi^{-1})(\tau)=e(\tau)\equiv
  \begin{cases}
    1 & \text{if} \;\tau=\emptyset, \\
    0 & \text{otherwise},
  \end{cases}
\end{equation}
and associativity is proved by
\begin{equation}
 B(\phi_z,B(\phi_y,B(\phi_x,x_0;h);h);h)
   = B(\phi_x\mlt (\phi_y\mlt \phi_z),x_0;h)
   = B((\phi_x\mlt \phi_y) \mlt \phi_z),x_0;h). 
   \label{bsa2:eq:ass_z}
   \end{equation}
This holds even if $\phi_z(\emptyset)\not\equiv1$.

The next result will be needed for the investigation of modified Newton iterations.
\begin{lem} \thlabel{bsa2:lem:Bfprime}
If $\phi_x(\emptyset)\equiv0$ we have
\[
\partial_2B(\phi_{y},{x_{0};h})B(\phi_{x},{x_{0};h})
=B(\phi_{x}\mltl\phi_{y},{x_{0};h}),
\]
where the bi-linear operator $\mltl$ is given by
\begin{equation} \label{bsa2:eq:mltl}
(\phi_x\mltl\phi_y)(\tau)=\begin{cases}
  0
  &if \; \tau=\emptyset,\\
  \sum\limits_{(\vartheta,\{\delta\})\in SP(\tau)}\gamma(\tau,\vartheta,\{\delta\})\cdot
  \phi_{y}(\vartheta)\phi_{x}
  (\delta) &\text{otherwise},
  \end{cases}
\end{equation}
with
\[
SP(\tau)=\{(\vartheta,\omega)\in ST(\tau):~\omega\text{ contains
exactly one element } \delta\}.\]
\end{lem}
\begin{proof}
Written in full, the statement of the theorem claims that
\begin{align}
  &\sum_{\vartheta\in T}\sum_{\delta\in T} \alpha(\vartheta)
  \alpha(\delta)\cdot \phi_y(\vartheta)\phi_x(\delta) \cdot
  \left(\partial F(\vartheta)F(\delta)\right) \nonumber \\
  &\qquad = \sum_{\tau\in T\setminus\{\emptyset\}} \alpha(\tau) \left(\sum_{(\vartheta,\{\delta\})
  \in SP(\tau)} \gamma(\tau,\vartheta,\{\delta\})\cdot \phi_y(\vartheta)\phi_x(\delta) \right) \cdot
  F(\tau). \label{bsa2:eq:dbb_full}
\end{align}
This is true if
\begin{align}
\left(\partial F(\vartheta)F(\delta)\right)& = \sum_{\tau\in \mathcal{A}(\vartheta,\delta)} \beta(\tau,\vartheta,\delta) F(\tau)\label{bsa2:eq:dbb_a} \\ \intertext{and}
\alpha(\vartheta)\alpha(\delta)\beta(\tau,\vartheta,\delta) &= \alpha(\tau)\gamma(\tau,\vartheta,\{\delta\})\label{bsa2:eq:dbb_b},
\end{align}
where $\mathcal{A}(\vartheta,\delta)$ is the set of all $\tau$'s constructed by attaching $\delta$ to one of the vertices of $\vartheta$. We will prove this by induction.

First, let $\vartheta=\emptyset$. Since $\partial F(\emptyset) F(\delta) = F(\delta)$ we have $\tau=\delta$ and \eqref{bsa2:eq:dbb_a} and \eqref{bsa2:eq:dbb_b} are trivially true with $\beta(\tau,\emptyset,\tau)=1$. Now, let $\vartheta=\bullet_l$. Then $\partial F(\vartheta)F(\delta)= g_l'F(\delta)=F([\delta]_l)$. As $\mathcal{A}(\vartheta,\delta)=\{[\delta]_l\}$ this gives $\beta([\delta]_l,\bullet_l,\delta)=1$, and again \eqref{bsa2:eq:dbb_b} is trivially true. Finally, let
\[ \vartheta=[\bar{\vartheta}_1^{r_1},\dotsc,
   \bar{\vartheta}_i^{r_i},\dotsc,\bar{\vartheta}_q^{r_q}]_l\]
   with distinct trees $\bar{\vartheta}_1,\dots,\bar{\vartheta}_q$. Then
\begin{align*}
 \partial F(\vartheta)F(\delta) &= g_l^{(\kappa_{\vartheta}+1)}
 \big(F(\delta),
 \overbrace{F(\bar{\vartheta}_1),\dotsc, F(\bar{\vartheta}_1)}^{r_1 \text{ times}},\dotsc,
 \overbrace{F(\bar{\vartheta}_q),\dotsc, F(\bar{\vartheta}_q)}^{r_q \text{ times}}
 \big)\\
   &+ \sum_{i=1}^q r_i g_l^{(\kappa_{\vartheta})} \big(
   F(\bar{\vartheta}_1),\dotsc,
   \partial F(\bar\vartheta_i)F(\delta),\overbrace{F(\bar{\vartheta}_i),\dotsc,
     F(\bar{\vartheta}_i)}^{r_i-1\text{ times}}, \dotsc,F({\bar\vartheta_q})\big),
\end{align*}
where $\kappa_{\vartheta}=\sum_{i=1}^qr_i$,
so $\tau \in \mathcal{A}(\vartheta,\delta)$ is either
\[ \tau= [\delta,\bar{\vartheta}_1^{r_1},\dotsc,
   \bar{\vartheta}_i^{r_{i}},\dotsc,\bar{\vartheta}_q^{r_q}]_l \quad \text{ with } \quad
   \beta(\tau,\vartheta,\delta)=1\]
or
\[ \tau= [\bar{\vartheta}_1^{r_1},\dotsc,
   \tau_i,\bar{\vartheta}_i^{r_{i}-1},\dotsc,\bar{\vartheta}_q^{r_q}]_l
\text{ with }\tau_i\in\mathcal{A}(\bar{\vartheta}_i,\delta)
    \text{ and }
    \beta(\tau,\vartheta,\delta)=r_i\beta(\tau_i,\bar{\vartheta}_i,\delta).\]
In the first case, if $\delta=\bar{\vartheta}_j$
for some $j$, then $\alpha(\tau)=\alpha(\vartheta)\alpha(\delta)/M$ and $\gamma(\tau,\vartheta,\{\delta\})=M$ with $M=r_j+1$. Otherwise the same is valid with $M=1$. So \eqref{bsa2:eq:dbb_b} holds.
In the second case, assume that our induction hypothesis \eqref{bsa2:eq:dbb_b} is true for all $\tau_i \in \mathcal{A}(\bar{\vartheta}_i,\delta)$. We obtain
\[ \alpha(\tau) = \frac{r_i}{M}
   \frac{\alpha(\tau_i)}{\alpha(\bar{\vartheta}_i)}\alpha(\vartheta)
    \quad \text{ and } \quad
   \gamma(\tau,\vartheta,\{\delta\})=M
   \gamma(\tau_i,\bar{\vartheta}_i,\{\delta\}) \]
   with $M=r_j+1$ if $\tau_i=\bar{\vartheta}_j$
for some $j$ and $M=1$ otherwise.
It follows that
\begin{align*}
\alpha(\tau)\gamma(\tau,\vartheta,\{\delta\}) &= r_i
   \frac{\alpha(\tau_i)}{\alpha(\bar{\vartheta}_i)}
   \gamma(\tau_i,\bar{\vartheta}_i,\{\delta\})\alpha(\vartheta)
   = \alpha(\vartheta) \alpha(\delta) r_i \beta(\tau_i,\bar{\vartheta}_i,\delta)
   \\&
   = \alpha(\vartheta) \alpha(\delta) \beta(\tau,\vartheta,\delta).
\end{align*}
\end{proof}

\subsection{B--series of the exact and the numerical solutions}
From the results of the previous subsection, it is possible to find the B--series of the exact and numerical solutions. Here, the proofs are only sketched, for details consult \cite{debrabant08bao,debrabantXXcos}.

\begin{theo} \thlabel{bsa2:thm:Bex} The solution $X(t_0+h)$ of \eqref{bsa2:SDE_new}
    can be written as a B--series $B(\varphi,x_0; h)$ with
    \begin{align*}
        \varphi(\emptyset)&\equiv1,&\varphi(\bullet_l)(h)&=W_l(h), &
        \varphi(\tau=[\tau_1,\dots,\tau_{\kappa}]_l)(h)&=\int_0^h
        \prod_{j=1}^{\kappa} \varphi(\tau_j)(s)\star dW_l(s).
    \end{align*}
\end{theo}
\begin{proof}
Write the exact solution as some B--series $X(t_0+h)=B(\varphi,x_0;h)$. As $\varphi(\emptyset)\equiv1$, apply \thref{bsa2:lem:f_y} to $g_l(X(t_0+h))$ and \thref{bsa2:def:trees,bsa2:def:eldiff,bsa2:def:bseries} to obtain
\begin{equation}
    \label{bsa2:eq:g-series}
    g_l(B(\varphi,x_0;h)) = \sum_{\tau\in T_l}\alpha(\tau)
    \cdot \varphi'_l(\tau)(h)\cdot F(\tau)(x_0)
\end{equation}
in which
\[ \varphi'_l(\tau)(h) =
\begin{cases}
    1 & \text{if } \tau=\bullet_{l}, \\
    \D \prod_{j=1}^{\kappa} \varphi(\tau_j)(h) & \text{if} \;
    \tau=[\tau_1,\dots,\tau_{\kappa}]_l \in T_l.
\end{cases} \]
Insert this into the SDE \eqref{bsa2:SDE} and compare term by term.
\end{proof}

\begin{theo} \thlabel{bsa2:thm:Bnum} The numerical solution $Y_1$ given by \eqref{bsa2:eq:STM} can be written as a B--series
    \[\qquad Y_1 = B(\Phi,x_0; h)\] with $\Phi$ recursively defined by
    \begin{subequations} \label{bsa2:eq:B_iter_Y}
        \begin{eqnarray}
\Phi(\emptyset) &\equiv& 1,\\
\Phi(\tau)&=& \Phi_{ex}(\tau)+(\Phi\mlt\Phi_{im})(\tau).
        \end{eqnarray}
    \end{subequations}
\end{theo}
\begin{proof}
Write $Y_1=B(\Phi,x_0;h)$ and insert this into \eqref{bsa2:eq:STM}. As $\Phi(\emptyset)=\Phi_{ex}(\emptyset)+\Phi_{im}(\emptyset)\equiv1$, apply \thref{bsa2:thm:comp}, and compare term by term.
\end{proof}

To study the consistency of the numerical methods, we need to assign to each tree an order:
\begin{defi}[Tree order]
    The order of a tree $\tau \in T$ respectively  $u\in U_f$ is defined by
    \[\rho(\emptyset)=0,\quad\rho(u=[\tau_1,\dots,\tau_\kappa]_f)=\sum\limits_{i=1}^\kappa\rho(\tau_i),\]
    and
    \[\rho(\tau=[\tau_1,\dots,\tau_\kappa]_l)=\sum\limits_{i=1}^\kappa\rho(\tau_i)+
    \begin{cases}
        1&\text{for }l=0,\\
        \frac12&\text{otherwise}.
    \end{cases}\]
\end{defi}
\begin{table}[t]
\centering
\[
\begin{array}{c@{}ccrl}
  \tau & \rho(\tau) & \alpha(\tau) & \multicolumn{2}{c}{\varphi(\tau)(h)} \\ \hline \mbox{} \\
  \tnr{l} &
  \begin{cases}
    1 & \!\! \text{if } l=0 \\
    \frac{1}{2} & \!\!\text{if } l\not=0
  \end{cases}
  & 1 & W_l(h) & =
  \begin{cases}
    h     & \text{if} \; l=0  \\
    J_{(l)} & (\text{S}) \\ I_{(l)} & \text{(I)}
  \end{cases} \\[6mm]
  \begin{matrix}
  \dtree{\tnr{1}}{\dtree{\tnr{0}}{\tnr{2}}} \\
   \end{matrix}
  & 2 & 1 & \int_0^h \int_0^{s_1}
  W_2(s_2)\star d s_2 \star dW_1(s_1)&=
  \begin{cases}
    J_{(2,0,1)} & \text{(S)} \\ I_{(2,0,1)} & \text{(I)}
  \end{cases} \\[6mm]
  \dtree{\tnr{0}}{\tnl{1} \tnr{1}} & 2 & \frac{1}{2} & \int_0^h
  W_1(s)^2\star ds
   & =
  \begin{cases}
    2J_{(1,1,0)}
    & (\text{S}) \\
     2I_{(1,1,0)} + I_{(0,0)}
     & (\text{I})
  \end{cases} \\[6mm]
  \raisebox{-3ex}{\dtree{\tnr{0}}{\tnl{1} \dtree{\tnr{1}}{\tnl{2} \tnr{2}}}}
  & 3 & \frac{1}{2} &
  \multicolumn{2}{l}{
  \int_{0}^h W_1(s_1)\left(\int_{0}^{s_1}W_2(s_2)^2\star
    dW_1(s_2)\right)\star ds_1 }\\ && \multicolumn{3}{r}{=
 \begin{cases}
    4J_{(2,2,1,1,0)}+2J_{(2,1,2,1,0)}+2J_{(1,2,2,1,0)}
     & (\text{S}) \\[2mm]
    4I_{(2,2,1,1,0)}+2I_{(2,1,2,1,0)}+2I_{(1,2,2,1,0)} & \\ +2I_{(0,1,1,0)}+2I_{(2,2,0,0)}+I_{(1,0,1,0)}+I_{(0,0,0)}
    & (\text{I})
 \end{cases}} \\ \hline
\end{array}
\]
\caption{\label{bsa2:tab:trees1} Examples of trees and corresponding
  functions $\rho(\tau)$, $\alpha(\tau)$, and $\varphi(\tau)$. The
  integrals $\varphi(\tau)$ are also expressed in terms of multiple
  integrals $J_{(\dots)}$ for the Stratonovich (S) and  $I_{(\dots)}$
  for the It\^{o} (I) cases, see \cite{kloeden99nso} for their
  definition. In bracket notation, the trees will be written as
  $\bullet_l$, $[[\bullet_2]_0]_1$, $[\bullet_1,\bullet_1]_0$, and
  $[\bullet_1,[\bullet_2,\bullet_2]_1]_0$, respectively.}
\end{table}
In Table \ref{bsa2:tab:trees1} some trees and the corresponding values for the functions $\rho$, $\alpha$, and $\varphi$ are presented.

To decide the weak order we will also need the B--series of the function $f$, evaluated at the exact and the numerical solution.
From \thref{bsa2:thm:Bex,bsa2:thm:Bnum} and \thref{bsa2:lem:f_y} we obtain
\[
f(X(t_0+h))=\sum_{u\in U_f}\beta(u)\cdot\psi_\varphi(u)(h)\cdot G(u)(x_0),
\]
\[
f(Y_1)=\sum_{u\in U_f}\beta(u)\cdot\psi_\Phi(u)(h)\cdot G(u)(x_0),
\]
with \[\psi_\varphi([\emptyset]_f)\equiv1,\quad
\psi_\varphi(u=[\tau_1,\dots,\tau_\kappa]_f)=
\prod\limits_{j=1}^\kappa \varphi(\tau_j),
\]
and \[\psi_\Phi([\emptyset]_f)\equiv1,\quad
\psi_\Phi(u=[\tau_1,\dots,\tau_\kappa]_f)=
\prod\limits_{j=1}^\kappa \Phi(\tau_j).\]

One can show \cite{kloeden99nso,burrage99rkm,debrabant10ste} that
 $E\psi_{\varphi}(u)(h)=\mathcal{O}(h^{\rho(u)})$
$\forall u\in U_f$ and $\varphi(\tau)(h)=\mathcal{O}(h^{\rho(\tau)})$ $\forall\tau\in T$,
 respectively, where especially in the latter case the $\mathcal{O}(\cdot)$-notation refers to the $L^2(\Omega)$-norm and $h\to0$.

In the following we assume that also method \eqref{bsa2:eq:STM} is consistent with the definition of the tree order, i.\,e.\ that it is constructed as usual such that
 $E\psi_{\Phi}(u)(h)=\mathcal{O}(h^{\rho(u)})$
$\forall u\in U_f$ and $\Phi(\tau)(h)=\mathcal{O}(h^{\rho(\tau)})$
$\forall\tau\in T$, respectively. These conditions are fulfilled if for $\tau\in T$ and $k\in\N=\{0,1,\dots\}$ it holds that
$(\Phi_{ex}(\tau))^{2^k}=\mathcal{O}(h^{2^k\rho(\tau)})$ and
$(\Phi_{im}(\tau))^{2^k}=\mathcal{O}(h^{2^k\rho(\tau)})$.

\section{B--series of the iterated solution and growth functions}\label{bsa2:sec:iter}
In this section we will discuss how the iterated solution defined in
\eqref{bsa2:eq:SRKiter}  can be written in terms of B--series, that is
\[Y_{1,k} = B(\Phi_k,x_0;h).  \]

Assume that the predictor can be written as a B--series,
\[ Y_{1,0} = B(\Phi_0,x_0;h), \]
satisfying $\Phi_0(\emptyset)\equiv1$  and $\Phi_0(\tau)=\bO(h^{\rho(\tau)})$ $\forall\tau\in T$.
The most common situation is the use of the trivial predictor
$Y_{1,0}=x_0$, for
which $\Phi_{0}(\emptyset)\equiv1$ and $\Phi_{0}(\tau)\equiv0$
otherwise.

We are now ready to study each of the iteration schemes, which differ only in the choice of $J_k$ in \eqref{bsa2:eq:SRKiter}.
In each case, we will first find the recurrence formula for
$\Phi_{k}(\tau)$. From this we define a growth function
$\mathfrak{g}(\tau)$:
\begin{defi}[Growth function] \thlabel{bsa2:def:growth}
    A growth function $\frg:T \rightarrow \N$  is a function
    satisfying
    \begin{equation}
        \label{bsa2:eq:growth}
    \begin{aligned}
    \Phi_{k}(\tau) & = \Phi(\tau) \quad \forall \tau
    \in T \text{ with }\frg(\tau) \leq k  \\
    \qquad \Rightarrow &  \qquad
    \Phi_{k+1}(\tau)  = \Phi(\tau) \quad \forall \tau
    \in T\text{ with } \frg(\tau) \leq k + 1,
    \end{aligned}
    \end{equation}
    for all $k\geq 0$.
\end{defi}
This result should be sharp in the sense that in general there exists $\tau \not= \emptyset$ with
$\Phi_{0}(\tau)\not=\Phi(\tau)$ and $\Phi_{k}(\tau)\not=\Phi(\tau)$ when $k<\mathfrak{g}(\tau)$.
From \thref{bsa2:lem:f_y} we also have
    \[ f(Y_{1,k})=\sum_{u\in U_f}\beta(u)\cdot\psi_{\Phi_{k}}{(u)}\cdot G(u)(x_0)
    \]
    with
    \[\psi_{\Phi_{k}}{([\emptyset]_f)}\equiv1,\quad
   \psi_{\Phi_{k}}{(u=[\tau_1,\dots,\tau_\kappa]_f)} = \prod\limits_{j=1}^\kappa
\Phi_{k}{(\tau_j)},\]
where $\beta(u)$ and $G(u)(x_{0})$ are given in \thref{bsa2:lem:f_y}.
This implies
\begin{equation}
    \label{bsa2:eq:growthPrime}
        \psi_{\Phi_{k}}(u) = \psi_{\Phi(\tau)} \quad \forall
        u=[\tau_{1},\dots,\tau_{\kappa}]_{f} \in U_{f} \text{ with }
         \mathfrak{g}'(u)=\max_{j=1}^{\kappa}\mathfrak{g}(\tau_{i})
        \leq k.
\end{equation}
As we will see, the growth functions give a precise description of the development of the iterations. However, to get applicable results we will at the end need the relation between the growth functions and the order. These aspects are discussed in the next section. Examples of trees and the values of the growth functions for the three iteration schemes are given in Figure \ref{bsa2:fig:t_and_g}.
\begin{figure}[t]
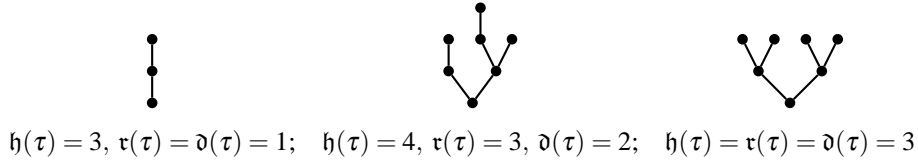

  \[
  \begin{array}{ccccc}
    \dtree{\tn}{\dtree{\tn}{\tn}}&
    \dtree{\tn}{\dtree{\tn}{\tn} \dtree{\tn}{\dtree{\tn}
      \tn \tn}} &
   \dtree{\tn}{\dtree{\tn}{\tn \tn} \dtree{\tn}{\tn \tn}}
    \\[2mm]
    \frh(\tau)=3,\; \frr(\tau)=\frd(\tau)=1;&
    \frh(\tau)=4,\;\frr(\tau)=3,\; \frd(\tau)=2; &
    \frh(\tau)=\frr(\tau)=\frd(\tau)=3
  \end{array}
  \]
  \caption{Examples of trees and their growth functions for simple
    ($\frh$), modified Newton ($\frr$) and full Newton ($\frd$) iterations. }
  \label{bsa2:fig:t_and_g}
\end{figure}
\subsection{The simple iteration}
Simple iterations are described by
\eqref{bsa2:eq:SRKiter}
with $J_k=0$, that is
\begin{equation} \label{bsa2:eq:simple}Y_{n+1,k+1}=B(\Phi_{ex},Y_n;h)+B(\Phi_{im},Y_{n+1,k};h).
\end{equation}
By \thref{bsa2:thm:comp} we easily get the following lemma, where, as in the following, all results are
valid for all $l=0,\dots,m$:
\begin{lem} \thlabel{bsa2:lem:B_H_iter} If $Y_{1,0} =
    B(\Phi_0,x_0;h)$ then $Y_{1,k} =
    B(\Phi_k,x_0;h)$, where
   \[
   \begin{aligned}
   \Phi_{k+1}(\emptyset) &\equiv 1, \\
   \Phi_{k+1}(\tau)&=\Phi_{ex}(\tau)+(\Phi_{k}\mlt\Phi_{im})(\tau).
   \end{aligned}
   \]
    The corresponding growth function is given by
   \[
   \mathfrak{h}(\emptyset)=0,\quad
   \mathfrak{h}(\bullet_{l})=1,\quad
   \mathfrak{h}([\tau_1,\dots,\tau_\kappa]_l)
   =1+\max\limits_{j=1}^\kappa\mathfrak{h}(\tau_j).
   \]
\end{lem}
The function $\frh(\tau)$ is the height of $\tau$, that is the maximum
number of nodes along one branch.

\subsection{The modified Newton iteration}
In this subsection we consider the modified Newton iteration
\begin{equation} \label{bsa2:eq:modified}
  \begin{aligned}
  Y_{n+1,k+1}&=B(\Phi_{ex},Y_n;h)+B(\Phi_{im},Y_{n+1,k};h) \\
  &+ \partial_2B(\Phi_{im},Y_n;h)(Y_{n+1,k+1}-Y_{n+1,k}).
  \end{aligned}
\end{equation}
The B--series for  $Y_{1,k}$ and the corresponding growth function can now be described by the following lemma:
\begin{lem} \thlabel{bsa2:lem:B_mod}
If $Y_{1,0} =  B(\Phi_0,x_0;h)$ then $Y_{1,k} = B(\Phi_k,x_0;h)$
with
\begin{equation} \label{bsa2:eq:lmod}
  \begin{aligned}
    \Phi_{k+1}(\emptyset)&\equiv1, \\
    \Phi_{k+1}(\tau)&=\Phi_{ex}(\tau)+(\Phi_{k}\mlt\Phi_{im})(\tau)+
    ((\Phi_{k+1}-\Phi_{k})\mltl\Phi_{im})(\tau).
  \end{aligned}
\end{equation}
The corresponding growth function is given by
\begin{equation*}
    \mathfrak{r}(\emptyset) =0, \qquad  \mathfrak{r}(\bullet_{l})=1,
     \qquad
    \mathfrak{r}(\tau=[\tau_{1},\dots,\tau_{\kappa}]_{l}) =
    \begin{cases}
        \mathfrak{r}(\tau_{1}) & \text{if} \quad \kappa=1, \\
        \displaystyle 1 + \max_{j=1}^{\kappa}\,
        \mathfrak{r}(\tau_{j}) & \text{if}
        \quad \kappa \geq 2.
    \end{cases}
\end{equation*}
\end{lem}
The function $\mathfrak{r}(\tau)$ is one plus the maximum number of
ramifications along any branch of the tree.
\begin{proof}
\thref{bsa2:thm:comp,bsa2:lem:Bfprime} imply \eqref{bsa2:eq:lmod}.
We next prove that $\mathfrak{r}$ is the appropriate growth function.
If $\mathfrak{r}(\tau)=0$ then $\tau=\emptyset$ and
$\Phi_0(\tau)=\Phi(\tau)$. Assume now that $\Phi_k(\tau)=\Phi(\tau)$ $\forall\tau$ with $\mathfrak{r}(\tau)\leq k$.
If $(\vartheta,\omega)\in
ST(\tau)\setminus SP(\tau)$ with $\mathfrak{r}(\tau)\leq k+1$, then $\forall\delta\in\omega$
it holds $\mathfrak{r}(\delta)\leq\mathfrak{r}(\tau)$ and therefore
$\Phi_{k}(\delta)=\Phi(\delta)$. So, by \eqref{bsa2:eq:lmod} we have
$\forall\tau$ with $\mathfrak{r}(\tau)\leq k+1$
\[
  \begin{aligned}
\Phi_{k+1}(\tau)&= \Phi_{ex}(\tau)+\sum_{(\vartheta,\omega)\in ST(\tau)\setminus SP(\tau)}\gamma(\tau,\vartheta,\omega) \cdot\Phi_{im}(\vartheta)\prod_{\delta\in\omega}\Phi(\delta)
\\&+\sum_{(\vartheta,\{\delta\})\in SP(\tau)}\gamma(\tau,\vartheta,\{\delta\}) \cdot\Phi_{im}(\vartheta)\Phi_{k+1}(\delta),
    \end{aligned}
\]
and by induction on the number of nodes of these trees we obtain that $\Phi_{k+1}(\tau)=\Phi(\tau)$ $\forall\tau$ with $\mathfrak{r}(\tau)\leq k+1$.
\end{proof}
\subsection{The full Newton iteration}
In this subsection we consider the full Newton iteration
\eqref{bsa2:eq:SRKiter} with $J_k = \partial_2B(\Phi_{im},Y_{1,k};h)$.
Extending the $\mltl$-operator to the case when its first operand does not vanish on the empty tree by
\[(\phi_x\mltl\phi_y)(\tau)=((\phi_x-\phi_x(\emptyset)e)\mltl\phi_y)(\tau)+\phi_x(\emptyset)\phi_y(\tau),\] it follows that the B--series for $Y_{1,k}$ and the corresponding growth function satisfy:
\begin{lem}\thlabel{bsa2:lem:B_fullNewton}
If $Y_{1,0} =
    B(\Phi_0,x_0;h)$ then $Y_{1,{k+1}} =
    B(\Phi_{k+1},x_0;h)$
         with
    \begin{equation}
    \begin{aligned}
\Phi_{k+1}(\emptyset)&\equiv 1,\\
\Phi_{k+1}(\tau)&= \Phi_{ex}(\tau)+\left(\Phi_k\mlt\left(\left(\Phi_k^{-1}\mlt\Phi_{k+1}\right)\mltl\Phi_{im}\right)\right)(\tau).
    \end{aligned}\label{bsa2:eq:phifullNewton}
    \end{equation}
The corresponding growth function is given by
     \[
\begin{aligned}
    \mathfrak{d}(\emptyset) &= 0, \qquad  \mathfrak{d}(\bullet_{l})=
    1, \\
    \mathfrak{d}(\tau=[\tau_{1},\dots,\tau_{\kappa}]_{l})& =
    \begin{cases}
        \max_{j=1}^{\kappa} \mathfrak{d}(\tau_{j}) & \text{if} \quad
        \gamma  = 1, \\
        \max_{j=1}^{\kappa} \mathfrak{d}(\tau_{j}) +1& \text{if} \quad
        \gamma  \geq 2,
    \end{cases}
\end{aligned}
\]
where $\gamma$ is the number of trees in $\tau_1,\dots,\tau_\kappa$ satisfying
$\mathfrak{d}(\tau_{i}) = \max_{j=1}^{\kappa} \mathfrak{d}(\tau_{j})$.
\end{lem}
The function $\frd$ is called the doubling index of $\tau$.
\begin{proof}
  Writing the iterations in terms of B--series, we get
  \begin{equation}\label{bsa2:eq:fullBiter}
  B(\Phi_{k+1},x_0;h)=B(\Phi_{ex},x_0;h)+B(\Phi_{im},Y_{1,k};h)
                     +\partial_2 B(\Phi_{im},Y_{1,k};h)B(\Delta\Phi_k,x_0; h)
  \end{equation}
  with $\Delta \Phi_k(\tau) = \Phi_{k+1}(\tau)-\Phi_k(\tau)$.
  Let $x_0=B(\Phi^{-1}_{k},Y_{1,k};h)$ so that
  \[
  B(\Delta\Phi_k,x_0;h)=B(\Delta\Phi_k,B(\Phi^{-1}_{k},Y_{1,k};h);h) =
  B(\Phi^{-1}_k\mlt \Delta\Phi_k,Y_{1,k};h).
  \]
  The use of \thref{bsa2:lem:Bfprime} followed by the use of \thref{bsa2:thm:comp}
  give the following result:
 \begin{align*}
  \partial_2B(\Phi_{im},Y_{1,k}; h)B(\Phi^{-1}_k\mlt\Delta\Phi_k,Y_{1,k};
  h) &= B((\Phi^{-1}_k\mlt\Delta\Phi_k)\mltl\Phi_{im},Y_{1,k}; h) \\
  &=
  B(\Phi_{k}\mlt((\Phi^{-1}_k\mlt\Delta\Phi_k)\mltl\Phi_{im}),x_0; h).
  \end{align*}
  The operator $\mltl$ is bilinear and $\mlt$ is linear from the right,
  thus
  \[\Phi_{k}\mlt((\Phi^{-1}_k\mlt\Delta\Phi_k)\mltl\Phi_{im}) =
  \Phi_{k}\mlt((\Phi^{-1}_k\mlt\Phi_{k+1})\mltl\Phi_{im})-\Phi_k\mlt\Phi_{im}\]
  and the first part of the theorem is proven by \eqref{bsa2:eq:fullBiter}.
  We will now prove the
  second part. Assume that $\Phi_{k}(\tau)=\Phi_{k+1}(\tau)=\Phi(\tau)$ for all $\tau$
  satisfying $\frd(\tau)\leq k$. This is  true for $k=0$ and
  $\tau=\emptyset$. Let $\Psi_k = \Phi_{k}^{-1}\mlt\Phi_{k+1}$
  and notice that by the assumption above, $\Psi_k(\tau)$ equals the unit element $e(\tau)$ if $\frd(\tau)\leq k$. Consider a tree $\tau$ where $\frd(\tau)=k+1$. For this tree we obtain
\begin{align} \nonumber
   (\Psi_k\mlt\Phi_{im})(\tau)&=
    \sum_{(\vartheta,\{\delta\})\in SP(\tau)}
    \gamma(\tau,\vartheta,\{\delta\})\Phi_{im}(\vartheta)
    \Psi_k(\delta)\\&
    +\sum_{(\vartheta,\omega)\in ST(\tau)\setminus
    SP(\tau)}\gamma(\tau,\vartheta,\omega)
    \Phi_{im}(\vartheta)\prod_{\delta\in\omega}\Psi_k(\delta)
    = (\Psi_k\mltl\Phi_{im})(\tau)
    \label{bsa2:eq:full_mult}
\end{align}
since the last sum of \eqref{bsa2:eq:full_mult} disappears: For each
    $(\vartheta,\omega)\in ST(\tau)\setminus SP(\tau)$ (if any) there is at
    least one $\delta\in\omega$ satisfying $\frd(\delta)\leq k$ and thereby
    $\Psi_k(\delta)=0$. In this case we obtain
     \[ (\Phi_k\mlt((\Phi_k^{-1}\mlt\Phi_{k+1})\mltl\Phi_{im}))(\tau) =
     (\Phi_{k}\mlt\Phi_{k}^{-1}\mlt\Phi_{k+1}\mlt\Phi_{im})(\tau) =
     (\Phi_{k+1}\mlt\Phi_{im})(\tau)\]
    by \eqref{bsa2:eq:ass_z}, so that
    \[
    \Phi_{k+1}(\tau)=\Phi_{ex}(\tau)+(\Phi_{k+1}\mlt\Phi_{im})(\tau). \]
    The theorem is completed by induction on the number of nodes of $\tau$ and on $k$.
\end{proof}

\section{General convergence results for iterated methods}\label{bsa2:sec:generalconvergenceresults}
Now we will relate the results of the previous section to the order of the overall scheme. We have weak consistency of order $p$ if and only if
\begin{equation}
    \label{bsa2:eq:T_eq}
    \E\psi_\Phi(u)(h)=\E\psi_\varphi(u)(h)+\bO(h^{p+1})\quad\forall u\in U_f\text{ with }\rho(u)\leq p+\frac12
\end{equation}
(\eqref{bsa2:eq:T_eq} slightly weakens conditions given in \cite{roessler06rta}), and
mean square global order $p$ if \cite{burrage04isr}
\begin{eqnarray*}
\Phi(\tau)(h)&=&\varphi(\tau)(h)+\bO(h^{p+\frac12})\quad\forall\tau\in T\text{ with }\rho(\tau)\leq p,\\
\E\Phi(\tau)(h)&=&\E\varphi(\tau)(h)+\bO(h^{p+1})\quad\forall\tau\in T\text{ with }\rho(\tau)\leq p+\frac12,
\end{eqnarray*}
and all elementary differentials $F(\tau)$ fulfill a linear growth condition. Instead of the last requirement it is also enough to claim that there exists a constant $C$ such that $\|g_j'(y)\|\leq C\quad\forall y\in\real^m$, $j=0,\dots,M$, and all necessary partial derivatives exist \cite{burrage00oco}.

Then, the order of
the iterated solution after $k$ iterations is $q_{k}$ if
\begin{equation}\label{bsa2:eq:iterWeakOrdCond} \E \psi_{\Phi_k}(u) = \E \psi_{\varphi}(u) \quad \forall u\in U_{f}\text{ with }
\rho(u) \leq q_{k}+\frac{1}{2} \end{equation}
in the weak convergence case respectively
\begin{equation}\label{bsa2:eq:iterStrOrdCond}
    \begin{aligned}
        \Phi_{k}(\tau)&=\varphi(\tau)\quad\forall\tau\in T\text{
            with
     }\rho(\tau)\leq q_{k},\\
        \E\Phi_{k}(\tau)&=\E\varphi(\tau)\quad\forall\tau\in T\text{ with }\rho(\tau)=q_{k}+\frac12
    \end{aligned}
\end{equation}
in the mean square convergence case.

In the following, we assume that the predictors satisfy the condition
\begin{equation} \label{bsa2:eq:cond_pred}
    \begin{aligned}
        \Phi_0(\tau)&= \Phi(\tau)
        \quad&\forall\tau\in T\text{ with }\mathfrak{g}(\tau)\leq
        \mathcal{G}_{0},
    \end{aligned}
\end{equation}
where $\mathcal{G}_{0}$
 is chosen as large
as possible. In particular, the trivial predictor satisfies
$\mathcal{G}_{0}=0$.

    It follows from \eqref{bsa2:eq:growth} and \eqref{bsa2:eq:growthPrime}
that
\begin{equation} \label{bsa2:eq:growthH}
    \begin{aligned}
        \Phi_k(\tau)&= \Phi(\tau) &&\forall\tau\in T\text{ with
        }\mathfrak{g}(\tau)\leq
        \mathcal{G}_{0}+k, \\
    \end{aligned}
\end{equation}
as well as
\begin{equation} \label{bsa2:eq:growthPhi}
    \begin{aligned}
        \psi_{\Phi_k}(u)&= \psi_{\Phi}(u) \quad&\forall u\in
        U_{f}\text{ with }\mathfrak{g}'(u)\leq
        \mathcal{G}_{0}+k.
    \end{aligned}
\end{equation}

The next step is to establish the relation between the order and the
growth function of a tree. We have chosen to do so by a maximum
growth function, given by
\begin{equation} \label{bsa2:eq:Gf}
    \begin{aligned}
      \mathcal{G}(q) &= \max_{\tau\in T}\left\{
      \mathfrak{g}(\tau) : \rho(\tau)\leq q \right\}= \max_{u\in U_f}\left\{
          \mathfrak{g}'(u) : \rho(u)\leq q\right\}.
    \end{aligned}
\end{equation}
With this definition, by \eqref{bsa2:eq:growthPhi} respectively \eqref{bsa2:eq:growthH},
the conditions \eqref{bsa2:eq:iterWeakOrdCond} respectively \eqref{bsa2:eq:iterStrOrdCond} are
 fulfilled for all $u$ of order $\rho(u)\leq
\min{(q_{k},p)}$ respectively all $\tau$ of order \mbox{$\rho(\tau)\leq
\min{(q_{k},p)}$} if
\begin{equation}\label{bsa2:eq:iterconvcond2}
        \mathcal{G}(q_{k}+\frac{1}{2}) \leq {\mathcal{G}}_{0}+k.
\end{equation}
Let $T^{S}\subset T$ and $U_{f}^{S} \subset U_{f}$ be the set of trees with an even number of each kind of stochastic nodes.
E.\,g.\ from \cite{debrabant10ste} we have
\begin{equation}
    \label{bsa2:eq:Eis0}
    \begin{aligned}
    \E \varphi(\tau) = 0 \quad &\text{if} \quad \tau \not\in T^{S},
    \\
    \E \psi_{\varphi}(u) = 0 \quad &\text{if} \quad u \not\in U_{f}^{S}.
    \end{aligned}
\end{equation}
Thus, if the method is as usual constructed such that
also $\forall m,n\in\N$ and $\forall\tau_{1,i}\in T$, $i=1,\dots,m$, $\forall\tau_{2,j}\in T$, $j=1,\dots,n$,
\begin{equation}\label{bsa2:eq:iteraddcond}
\E\left(\prod_{i=1}^m\prod_{j=1}^n\Phi_{ex}(\tau_{1,i})\Phi_{im}(\tau_{2,j})\right)=0\qquad\text{if}\qquad\sum_{i=1}^m\rho(\tau_{1,i})+\sum_{j=1}^n\rho(\tau_{2,i})\notin\N,
\end{equation}
then in \eqref{bsa2:eq:iterconvcond2} $q_k+\frac12$ can be replaced by $\lfloor q_k+\frac12\rfloor.$

The results can then be summarized in the following theorem:
\begin{theo}\thlabel{bsa2:lem:C}\sloppypar
If \eqref{bsa2:eq:iteraddcond} is fulfilled, then the iterated method is of weak respectively mean square
     order
    $q_{k}\leq p$ after
$
        \mathcal{G}(\lfloor q_{k}+\frac{1}{2}\rfloor)- {\mathcal{G}}_{0}
$
    iterations, otherwise
after
\mbox{$\mathcal{G}(q_{k}+\frac{1}{2})- {\mathcal{G}}_{0}$}
    iterations.
\end{theo}

Our next aim is to give explicit formulas for the maximum growth function.
 Let us start with the following lemma.
\begin{lem} \thlabel{bsa2:lem:GtoOrd}
    For $k \geq 1$,
    \[
    \begin{aligned}
        \mathfrak{h}(\tau)=k & \quad \Rightarrow \quad \rho(\tau)\geq
        \frac{k}{2},  \\
        \mathfrak{r}(\tau)=k & \quad \Rightarrow \quad \rho(\tau) \geq
        k-\frac12,
        \\
        \mathfrak{d}(\tau)=k & \quad \Rightarrow \quad \rho(\tau) \geq
        2^{k-1}-\frac12.
    \end{aligned}
    \]
The same result is valid for  $\mathfrak{h}'(u)$, $\mathfrak{r}'(u)$,
and $\mathfrak{d}'(u)$.
\end{lem}
\begin{proof}
    Let $\Tau_{\frh,k}$, $\Tau_{\frr,k}$, and $\Tau_{\frd,k}$ be sets of trees of minimal order satisfying $\frh(\tau)=k$ $\forall\tau\in\Tau_{\frh,k}$,
    $\frr(\tau)=k$ $\forall\tau\in\Tau_{\frr,k}$, and $\frd(\tau)=k$
    $\forall\tau\in\Tau_{\frd,k}$ (see Figure \ref{bsa2:fig:minordtrees}), and denote this minimal order by
    $\rho_{\frh,k}$,  $\rho_{\frr,k}$, and  $\rho_{\frd,k}$.
    Minimal order trees are build up only by stochastic nodes. It follows immediately that
    $\Tau_{\frh,1}=\Tau_{\frr,1}=\Tau_{\frd,1}=\{\bullet_{l}:~l\geq1\}$. Since
    $\rho(\bullet_{l})=1/2$ for $l\geq 1$, the results are proved for $k=1$. It is
    easy to show by induction on $k$ that
    \begin{equation} \label{bsa2:eq:MinOrd}
    \begin{aligned}
        \Tau_{\frh,k}&=\{[\tau]_{l}:~\tau\in\Tau_{\frh,k-1},~l\geq1\},\quad
        &\rho_{\frh,k}&=\rho_{\frh,k-1}+\frac{1}{2} =
        \frac{k}{2}, \\
            \Tau_{\frr,k}&=\{[\bullet_{l_1},\tau]_{l_2}:~\tau\in\Tau_{\frr,k-1},~l_1,l_2\geq1\}, \quad &
        \rho_{\frr,k}&=\rho_{\frr,k-1}+1 =
        k-\frac{1}{2}, \\
        \Tau_{\frd,k}&=\{[\tau_1,\tau_2]_{l}:~\tau_1,\tau_2\in\Tau_{\frd,k-1},~l\geq1\}, \quad &
        \rho_{\frd,k}&=2\rho_{\frd,k-1}+\frac{1}{2} = 2^{k-1}-\frac{1}{2}.
    \end{aligned}
    \end{equation}
For each $\frg$ being either $\frh$, $\frr$, or $\frd$,
the minimal order trees satisfying $\frg'(u_{\frg,k})=k$ are $u_{\frg,k}=[\tau_{\frg,k}]_{{f}}$ with $\tau_{\frg,k}\in\Tau_{\frg,k}$, which are of order $\rho(\tau_{\frg,k})$.
\end{proof}

\begin{figure}[t]
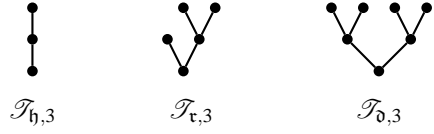

  \centering
  \[
    \begin{array}{ccccc}
       \dtree{\tn}{\dtree{\tn}{\tn}} &\qquad &
       \dtree{\tn}{\tn \dtree{\tn}{\tn \tn}} & \qquad &
       \dtree{\tn}{\dtree{\tn}{\tn \tn} \dtree{\tn}{\tn \tn}} \\[2mm]
       \Tau_{\frh,3} && \Tau_{\frr,3} && \Tau_{\frd,3}
    \end{array}
  \]
  \caption{Minimal order trees with $\frg(\tau)=3$. The sets
    $\Tau_{\frg,3}$ consist of all such trees with only stochastic
    nodes. }
  \label{bsa2:fig:minordtrees}
\end{figure}
Now we can prove the following corollary.
\begin{cor}\thlabel{bsa2:cor:growthfunctions} For $q \geq {\frac12}$ we have
    \[
    \mathcal{G}(q)  =
    \begin{cases}
        {2q} & \text{for simple iterations,} \\
        \lfloor q+\frac12 \rfloor & \text{for modified Newton iterations,} \\
        {\lfloor \log_{2}(q+\frac12) \rfloor +1}& \text{for full Newton iterations.}
    \end{cases}
    \]
\end{cor}
\begin{proof}
The minimal order trees are also the maximum height / ramification number / doubling index trees,
in the sense that as long as $\rho(\tau_{\frg,k}) \leq q <
\rho(\tau_{\frg,k+1})$ there are no trees of order $q$ for
which the growth function can exceed $k$.
\end{proof}

For some methods, these results can be refined. We call a method semi-implicit, if $\Phi_{im}(\tau)\equiv0$ $\forall\tau\notin T_0$ (remember that $T_0$ is the set of trees with a deterministic root). Then, by \thref{bsa2:lem:B_H_iter,bsa2:lem:B_mod,bsa2:lem:B_fullNewton} we obtain the following lemma:
\begin{lem} \thlabel{bsa2:lem:H_iter_semi}
For semi-implicit methods, the corresponding growth functions are given by
   \begin{align*}
   \mathfrak{h_s}(\emptyset)&=0,\quad
   \mathfrak{h_s}([\tau_1,\dots,\tau_\kappa]_l)=
   \begin{cases}
   1&\text{ if }l>0,\\
   1+\max\limits_{j=1}^\kappa\mathfrak{h_s}(\tau_j)&\text{ if }l=0,
   \end{cases}\\
    \mathfrak{r_s}(\emptyset) &=0, \quad  \mathfrak{r_s}(\bullet_{l})=1,
     \quad
    \mathfrak{r_s}(\tau=[\tau_{1},\dots,\tau_{\kappa}]_{l}) =
    \begin{cases}
        1&\text{if} \quad l>0,\\
        \mathfrak{r_s}(\tau_{1}) & \text{if} \quad l=0,\kappa=1,\\
        \displaystyle 1 + \max_{j=1}^{\kappa}\,
        \mathfrak{r_s}(\tau_{j}) & \text{if}
        \quad l=0, \kappa \geq 2,
    \end{cases}\\
    \mathfrak{d_s}(\emptyset) &= 0, \quad  \mathfrak{d_s}(\bullet_{l})=
    1,\quad
    \mathfrak{d_s}(\tau=[\tau_{1},\dots,\tau_{\kappa}]_{l}) =
    \begin{cases}
            1&\text{if}\quad l>0,\\
        \max\limits_{j=1}^{\kappa} \mathfrak{d_s}(\tau_{j}) & \text{if} \quad
        l=0, \gamma  = 1, \\
        \max\limits_{j=1}^{\kappa} \mathfrak{d_s}(\tau_{j}) +1& \text{if} \quad
          l=0,\gamma\geq 2,
    \end{cases}
   \end{align*}
where $\gamma$ is the number of trees in $\tau_1,\dots,\tau_\kappa$ satisfying
$\mathfrak{d_s}(\tau_{i}) = \max_{j=1}^{\kappa} \mathfrak{d_s}(\tau_{j})$.
\end{lem}
This implies immediately:
\begin{lem} \thlabel{bsa2:lem:GtoOrdsemi}
    For $k \geq 1$,
    \[
    \begin{aligned}
        \mathfrak{h_s}(\tau)=k & \quad \Rightarrow \quad \rho(\tau)\geq
        k-\frac12,\\
        \mathfrak{r_s}(\tau)=k & \quad \Rightarrow \quad \rho(\tau)\geq
        \frac32k-1,\\
        \mathfrak{d_s}(\tau)=k & \quad \Rightarrow \quad \rho(\tau)\geq
        \frac342^k-1.
\end{aligned}
\]
The same result is valid for $\mathfrak{h_s}'(u)$, $\mathfrak{r_s}'(u)$, and $\mathfrak{d_s}'(u)$.
\end{lem}
\begin{cor}
For semi-implicit methods we have for $q \geq {\frac12}$
    \[
    \mathcal{G}(q)=
        \begin{cases}
                \lfloor q+\frac12 \rfloor & \text{for simple iterations,} \\
        \lfloor \frac23(q+1) \rfloor & \text{for modified Newton iterations,} \\
        {\lfloor \log_{2}\frac{q+1}3 \rfloor+2 }& \text{for full Newton iterations.}
    \end{cases}
\]
\end{cor}
For the trivial predictor, Table \ref{bsa2:table1} gives the number of iterations needed to achieve a certain order of convergence, both in the general and in the semi-implicit case.
\begin{table}
\begin{center}
\begin{tabular}{c||c|c|c}
$p$&simple iter.&mod.\ iter.&full iter.\\\hline
$\frac12$&2 (1)&1&1\\
1&2 (1)&1&1\\
$1\frac12$&4 (2)&2 &2\\
2&4 (2)&2&2\\
$2\frac12$&6 (3)&3 (2)&2\\
3&6 (3)&3 (2)&2
\end{tabular}
\end{center}
\caption{\label{bsa2:table1}Number of iterations needed to achieve order $p$ when using the simple, modified or full Newton iteration scheme in the It\^{o} and Stratonovich case for strong or weak approximation, provided \eqref{bsa2:eq:iteraddcond} is fulfilled. In parentheses, the numbers for semi-implicit methods are given.}
\end{table}

For the sake of completeness, we also give the corresponding results for (deterministic) Taylor methods applied to deterministic problems. Note that in this case, \eqref{bsa2:eq:iteraddcond} is automatically fulfilled.
\begin{lem} \thlabel{bsa2:lem:GtoOrddet}
Suppose that the considered problem is purely deterministic, i.\,e.\ $m=0$ in \eqref{bsa2:SDE}. Then, for $k \geq 1$,
    \[
    \begin{aligned}
        \mathfrak{h}(\tau)=k & \quad \Rightarrow \quad \rho(\tau)\geq
        k,\\
        \mathfrak{r}(\tau)=k & \quad \Rightarrow \quad \rho(\tau)\geq
        2k-1,\\
        \mathfrak{d}(\tau)=k & \quad \Rightarrow \quad \rho(\tau)\geq
        2^k-1.
\end{aligned}
\]
The same result is valid for $\mathfrak{h}'(u)$, $\mathfrak{r}'(u)$, and $\mathfrak{d_s}'(u)$.
\end{lem}
\begin{cor}
For deterministic problems, we have for $q\in\N$, $q\geq1$
    \[
    \mathcal{G}(q)=
        \begin{cases}
                q & \text{for simple iterations,} \\
        \lfloor \frac{q+1}2 \rfloor & \text{for modified Newton iterations,} \\
        {\lfloor \log_{2}(q+1) \rfloor}& \text{for full Newton iterations.}
    \end{cases}
\]
\end{cor}
\section{Numerical examples}\label{bsa2:sec:numerics}
In the following, we analyze numerically the order of convergence of several stochastic Taylor methods in dependence on the kind and number of iterations.

As first examples, we apply the semi-implicit Milstein method \cite{kloeden99nso}, denoted by SIM and given by \eqref{bsa2:eq:sim},
the implicit Milstein-Taylor method \cite{tian01itm}, denoted by IM and given by
\begin{equation*}
Y_{n+1}=Y_{n}+hg_0(Y_{n+1})+I_{(1)}g_1(Y_{n+1})-(I_{(1,1)}+h)[g_1'g_1](Y_{n+1}),
\end{equation*}
both of strong order 1.0,
and the semi-implicit strong order 1.5 Taylor method due to Kloeden and Platen \cite{kloeden99nso,tian01itm}, denoted by SIKP and given by
\begin{align*}
Y_{n+1}=&Y_{n}+hg_0(Y_{n+1})+I_{(1)}g_1(Y_{n})+I_{(1,1)}[g_1'g_1](Y_{n})
-I_{(0,1)}[g_0'g_1](Y_{n})
\\&-\frac12h^2[g_0'g_0+\frac12g_0''g_1^2](Y_{n+1})
+I_{(0,1)}[g_1'g_0+\frac12g_1''g_1^2](Y_{n})\\
&+I_{(1,1,1)}[g_1'^2g_1+g_1''g_1^2](Y_{n}),
\end{align*}
to the non-linear SDE \cite{kloeden99nso}
\begin{equation} \label{bsa2:eq:nonlinex}
    dX(t) = \left( \tfrac{1}{2} X(t) + \sqrt{X(t)^2 + 1} \right) \,
    dt + \sqrt{X(t)^2 + 1} \, dW(t), \qquad X(0)=0,
\end{equation}
on the time interval $I=[0,1]$ with the solution $X(t) = \sinh (t + W(t))$. With each method, the solution is approximated with step sizes $2^{-11}, \ldots, 2^{-15}$ and the sample average of $M=4000$ independent simulated realisations of the absolute error is calculated in order to estimate the expectation.

The results at time $t=1$ are presented in Figure \ref{bsa2:fig:3}, where the orders of convergence correspond to the slope of the regression lines. As predicted by Table \ref{bsa2:table1} we observe strong order 1.0 for one simple or one (modified) Newton iteration of the semi-implicit Milstein method; and no convergence for one simple iteration but strong order 1.0 for two simple or one (modified) Newton iteration of the implicit Milstein-Taylor method. The semi-implicit strong order 1.5 Taylor method yields strong order 1.0 for one and strong order 1.5 for two simple or modified Newton iterations.
\begin{figure}[tbp]
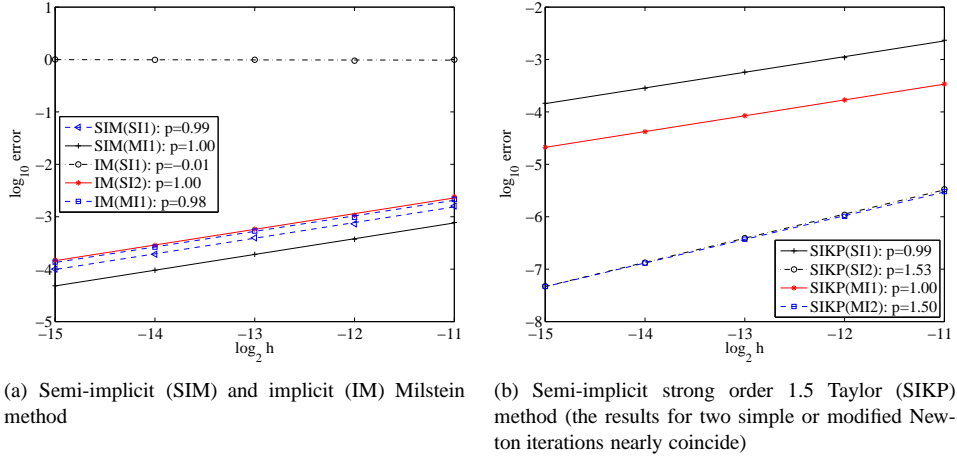

\subfigure[Semi-implicit (SIM) and implicit (IM) Milstein method]{\includegraphics[width=\figwidth]{\Farbbild{Bsp53SIMIM_3}}}
\hspace*{\fill}
\subfigure[Semi-implicit strong order 1.5 Taylor (SIKP) method (the results for two simple or modified Newton iterations nearly coincide)]{\includegraphics[width=\figwidth]{\Farbbild{Bsp53SIKP_3}}}
\caption{\label{bsa2:fig:3}Error of several Taylor methods applied to  \eqref{bsa2:eq:nonlinex} with up to two simple (SI) and modified Newton (MI) iterations}
\end{figure}

Next, we apply the semi-implicit weak order two Taylor scheme due to Platen \cite{kloeden99nso},
 denoted by SIW and given by
\begin{align*}
Y_{n+1}=&Y_{n}+hg_0(Y_{n+1})+I_{(1)}g_1(Y_{n})+I_{(1,1)}[g_1'g_1](Y_{n})
\\&+\frac12I_{(1)}h[-g_0'g_1+g_1'g_0+\frac12g_1''g_1^2](Y_{n})-\frac12h^2[g_0'g_0+\frac12g_0''g_1^2](Y_{n+1}),
\end{align*}
to SDE \eqref{bsa2:eq:nonlinex}. Here, we choose as functional $f(x)=p(\arsinh(x))$, where $p(z) = z^3 - 6z^2 + 8z$ is a polynomial. Then the expectation of the solution can be calculated as
\begin{equation}
     \E(f(X(t))) = t^3 - 3t^2 + 2t \,\,.
\end{equation}
The solution $\E(f(X(t)))$ is approximated with step sizes $2^{-3}, \ldots, 2^{-6}$ and $M=4\cdot10^9$ simulations are performed in order to determine the systematic error of SIW at time $t=1$.
The results with one or two simple or modified Newton iteration steps are presented in Figure \ref{bsa2:fig:1}. According to Table \ref{bsa2:table1} we expect approximation order one for one iteration and order two for two iterations, which is approved by Figure \ref{bsa2:fig:1}.
 \begin{figure}[tbp]
 \begin{center}
 \includegraphics[width=\figwidth]{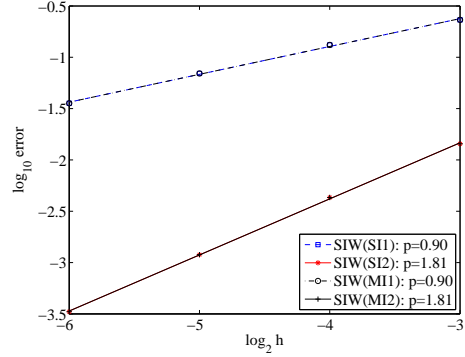}
 \end{center}
 \caption{\label{bsa2:fig:1}Error of the weak second order (semi-) implicit Platen method applied to \eqref{bsa2:eq:nonlinex} with one or two simple (SI) or modified Newton (MI) iterations (the results for one respectively two simple and modified Newton iterations coincide)}
 \end{figure}
Finally, we apply the fully implicit strong order 1.5 Taylor scheme given in \cite{debrabantXXcos},
\begin{align*}
Y_{n+1}=&Y_n+\frac12\I1g_{1,n+1}+\frac12hg_{0,n+1}+\frac12(\I{1,1}+h)g_{1,n+1}'g_{1,n+1}
+\frac14{h^2}g_{0,n+1}'g_{0,n+1}
\\&+\frac18{h^2}g_{0,n+1}''(g_{1,n+1},g_{1,n+1})
+\frac12\I1\g+\frac12h\f
-\left(h+\frac12\I{1,1}\right)\g'\g
\\&+\frac12\left(\I{0,1}-\I{1,0}\right)\g'\f
-\frac12(\I{0,1}-\I{1,0})\f'\g
\\&+\left(\frac12\I{0,1}-\frac74h\I1
-2\I{1,1,1}\right)\g''(\g,\g)
-\left(\frac32h\I1+2\I{1,1,1}\right)\g'\g'\g
\\&-\frac14{h^2}\f'\f
-h^2\g''(\f,\g)-\frac14{h^2}\g'\f'\g
-\frac34h^2\g'\g'\f
-\frac18h^2\f''(\g,\g)
\\&-\frac14{h^2}\g'\g'\g'\g
-\frac58h^2\g'\g''(\g,\g)
-\frac74h^2\g''(\g'g,\g)
-\frac34h^2\g'''(\g,\g,\g)
\end{align*}
(here we used the abbreviations $g_{l,n+1}=g_l(Y_{n+1})$ and $g_{l}=g_l(Y_{n})$),
which is denoted by FIT, and the semi-implicit strong order 1.5 scheme SIKP to the system of non-linear SDEs
\begin{equation}
\label{bsa2:eq:nonlinex2d}
\begin{split}
dX_1(t)&=\left(\frac12X_1(t)+\sqrt{X_1(t)^2+X_2(t)^2+1}\right)~dt+\big(\sin(X_1(t))+2\sin(X_2(t))\big)~dW(t),\\
dX_2(t)&=\left(\frac12X_1(t)+\sqrt{X_2(t)^2+1}\right)~dt+\big(\cos(X_1(t))+3\cos(X_2(t))\big)~dW(t),\\
X_1(0)&=0,\quad X_2(0)=0,
\end{split}
\end{equation}
again on the time interval $I=[0,1]$. The solution is approximated with step sizes $2^{-11}, \ldots, 2^{-15}$ and the sample average of $M=4000$ independent simulated realisations of the absolute error is calculated in order to estimate the expectation. As here we do not know the exact solution, to approximate it we use SIKP with two simple iterations and a step size ten times smaller than the actual step size.

The numerical results at $t=1$ are presented in Figure \ref{bsa2:fig:2}. Again, the orders expected according to Table \ref{bsa2:table1} are confirmed. \begin{figure}[tbp]
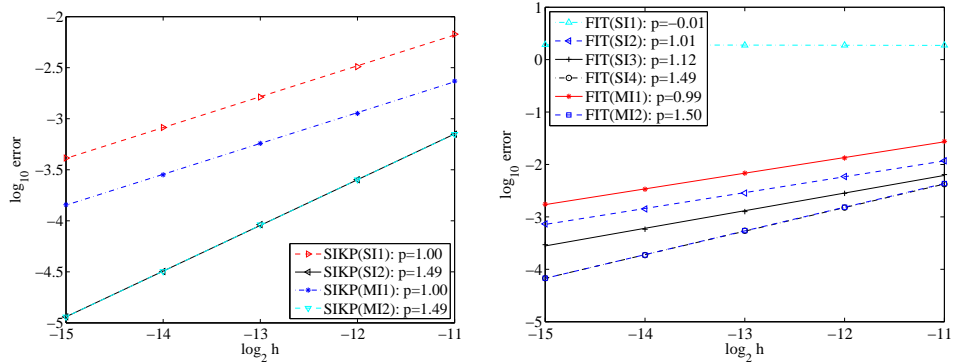

\subfigure[Semi-implicit strong order 1.5 Taylor (SIKP) method (the results for two simple or modified Newton iterations coincide)]{\includegraphics[width=\figwidth]{\Farbbild{Bsp64SIKP_3}}}
\hspace*{\fill}
\subfigure[Full implicit strong order 1.5 Taylor (FIT) method (the results for four simple or two modified Newton iterations nearly coincide)]{\includegraphics[width=\figwidth]{\Farbbild{Bsp64FIT_3}}}
 \caption{\label{bsa2:fig:2}Error of SIKP and FIT applied to \eqref{bsa2:eq:nonlinex2d} with different numbers of simple (SI) and modified Newton (MI) iterations}
 \end{figure}

\section{Conclusion}
For stochastic implicit Taylor methods that use an iterative scheme to approximate the solution, we derived stochastic B--series and corresponding growth functions. From these, we deduced convergence results based on the order of the underlying Taylor method, the choice of the iteration method, the predictor, and the number of iterations, for It\^{o} and Stratonovich SDEs, and for weak as well as strong convergence. The convergence results are confirmed by numerical experiments. From a practical point of view, this theory might lead to the construction of more efficient numerical schemes for SDEs. But we also like to point out that the similarities of the iteration dependent growth functions $\frg$ for a range of problems (ODEs, DAEs, and SDEs) and underlying methods (Runge--Kutta methods, implicit Taylor methods) indicate an underlying structure that could well be investigated in a more general fashion. In spite of this, the number of iterations needed to obtain the order of the underlying implicit Taylor method does not depend on whether the SDE is of It\^{o} or Stratonovich type. This is in contrast to the results obtained for Runge--Kutta methods for SDEs, for which usually less iterations are needed in the Stratonovich case \cite{debrabant08bao}. The reason for this is that in the latter case certain error terms have vanishing expectation even if they do not vanish themselves.  This is not the situation for implicit Taylor methods.

\section{Acknowledgement} {We thank Professor Martin Arnold and Professor R\"{u}diger Weiner at the Arbeitsgruppe Numerische Mathematik, Martin-Luther-Universit\"{a}t Halle-Wittenberg, for their kind support during the final stage of this work.} 

\def\cprime{$'$}

\end{document}